\newfont{\bcb}{msbm10}
\newfont{\matb}{cmbx10}
\newfont{\got}{eufm10}

\documentclass[12pt]{amsart}
\usepackage{amsmath, amsthm, amscd, amsfonts, amssymb, latexsym, graphicx, color}
\usepackage[bookmarksnumbered, colorlinks, plainpages, hypertex]{hyperref}

\usepackage[cp1250]{inputenc}

\usepackage{amsmath,amsthm}
\usepackage{amssymb,latexsym}
\usepackage{enumerate}

\newtheorem{theorem}{Theorem}[section]

\newtheorem{proposition}[theorem]{Proposition}
\newtheorem{corollary}[theorem]{Corollary}
\theoremstyle{definition}

\theoremstyle{remark}
\newtheorem{remark}[theorem]{Remark}
\numberwithin{equation}{section}

\begin{document}

\title[H\"{o}lder and Lipschitz continuity of definable functions]{H\"{o}lder and Lipschitz continuity
       \\ of functions definable over
       \\ Henselian rank one valued fields}

\author[Krzysztof Jan Nowak]{Krzysztof Jan Nowak}






\begin{abstract}
Consider a Henselian rank one valued field $K$ of
equi\-characteristic zero with the three-sorted language
$\mathcal{L}$ of Denef--Pas. Let $f: A \to K$ be a continuous
$\mathcal{L}$-definable (with parameters) function on a closed
bounded subset $A \subset K^{n}$. The main purpose is to prove
that then $f$ is H\"{o}lder continuous with some exponent $s\geq
0$ and constant $c \geq 0$; a fortiori, $f$ is uniformly
continuous. Further, if $f$ is locally Lipschitz continuous with a
constant $c$, then $f$ is (globally) Lipschitz continuous with
possibly some larger constant $d$. Also stated are some problems
concerning continuous and Lipschitz continuous functions definable
over Henselian valued fields.
\end{abstract}

\maketitle


\section{Introduction}

Consider a Henselian rank one valued field $K$ of
equi\-characteristic zero along with the language $\mathcal{L}$ of
Denef--Pas, which consists of three sorts: the valued field
$K$-sort, the value group $\Gamma$-sort and the residue field
$\Bbbk$-sort. The only symbols of $\mathcal{L}$ connecting the
sorts are the following two maps from the main $K$-sort to the
auxiliary $\Gamma$-sort and $\Bbbk$-sort: the valuation map $v$
and an angular component map $\overline{ac}$ which is
multiplicative, sends $0$ to $0$ and coincides with the residue
map on units of the valuation ring $R$ of $K$. The language of the
$K$-sort is the language of rings; that of the $\Gamma$-sort is
any augmentation of the language of ordered abelian groups (with
$\infty$); finally, that of the $\Bbbk$-sort is any augmentation
of the language of rings. Throughout the paper the word
"definable" means "definable with parameters" and $K$-sort
variables are $x,y,z,\ldots$. We consider $K^{n}$ with the product
topology, called the $K$-topology on $K^{n}$, and adopt the
following convention:
$$ |x| = |(x_{1},\ldots,x_{n})| := \max \: \{ |x_{1}|, \dots,
   |x_{n}| \}
$$
and
$$ v(x) = v(x_{1},\ldots,x_{n}) := \min \: \{ v(x_{1}), \dots,
   v(x_{n}) \}
$$
for $x = (x_{1},\ldots,x_{n}) \in K^{n}$.

\vspace{1ex}

The main purpose of this paper is to prove the following two
theorems on H\"{o}lder and Lipschitz continuity.

\begin{proposition}\label{Hol}
Let $f: A \to K$ be a continuous $\mathcal{L}$-definable function
$f: A \to K$ on a closed bounded subset $A \subset K^{n}$. Then
$f$ is H\"{o}lder continuous with some exponent $r\geq0$ and
constant $c \geq0$, i.e.\
$$ |f(x) - f(z)| \leq c |x-z|^{r} $$
for all $x,z \in A$.
\end{proposition}

\begin{proposition}\label{Lip}
Let $f: A \to K$ be an $\mathcal{L}$-definable (with parameters)
function $f: A \to K$ on a closed bounded subset $A \subset
K^{n}$. Suppose $f$ is locally Lipschitz continuous with a
constant $c\geq 0$, i.e.\ each point $a \in A$ has a neighbourhood
$U$ such that
$$ |f(x) - f(z)| \leq c |x-z| \ \  \text{for all} \ \ x,z \in U. $$
Then $f$ is (globally) Lipschitz continuous with possibly some
larger constant $d$.
\end{proposition}

An immediate consequence of Proposition~\ref{Hol} is the following

\begin{corollary}
Every continuous $\mathcal{L}$-definable function $f: A \to K$ on
a closed bounded subset $A \subset K^{n}$ is uniformly continuous.
\end{corollary}

Proposition~\ref{Hol} follows immediately from a version of the
\L{}ojasiewicz inequality established in Section~2, which
generalizes the version from our
paper~\cite[Proposition~9.1]{Now}. Nevertheless, its formulation
is more classical in comparison with the latter and its proof
given in Section~2 follows similar arguments. Note that one of its
basic ingredients is the closedness theorem from our
paper~\cite[Theorem~3.1]{Now}. Proposition~\ref{Lip} relies
directly on the closedness theorem. Section~3 contains the proofs
of the two main theorems.

\begin{remark}
Proposition~\ref{Hol} is a counterpart of a well known theorem
that every continuous semi-algebraic or subanalytic function on a
compact subset of $\mathbb{R}^{n}$ is H\"{o}lder continuous. That
theorem remains valid for continuous functions which are definable
in polynomially bounded o-minimal structures (cf.~\cite{Dries-M}).
We thus see that, in a sense, also in algebraic geometry over
Henselian rank one valued fields $K$, closed bounded subsets of
$K^{n}$ correspond to compact subsets of $\mathbb{R}^{n}$.
\end{remark}

A different, more delicate problem is to examine locally Lipschitz
continuous definable functions on subsets that are not closed.
Then the conclusion of global Lipschitz continuity must be
replaced by that of piecewise Lipschitz continuity. In the last
Section~4, we refer to this problem over Henselian rank one valued
fields and also state two extension problems: of Lipschitz
continuous rational functions defined on algebraic subvarieties of
$K^{n}$ as well as of continuous and Lipschitz continuous
$\mathcal{L}$-definable functions defined on closed subsets of
$K^{n}$.

\vspace{1ex}

Finally note that not all valued fields $K$ have an angular
component map, but it exists if $K$ has a cross section, which
happens whenever $K$ is $\aleph_{1}$-saturated
(cf.~\cite[Chap.~II]{Ch}). Moreover, a valued field $K$ has an
angular component map whenever its residue field $\Bbbk$ is
$\aleph_{1}$-saturated (cf.~\cite[Corollary~1.6]{Pa2}). In
general, unlike for $p$-adic fields and their finite extensions,
adding an angular component map does strengthen the family of
definable sets. For both $p$-adic fields (Denef~\cite{De}) and
Henselian equicharacteristic zero valued fields (Pas~\cite{Pa1}),
quantifier elimination was established by means of cell
decomposition and a certain preparation theorem (for polynomials
in one variable with definable coefficients) combined with each
other. In the latter case, however, cells are no longer finite in
number, but parametrized by residue field variables.

\vspace{1ex}

Nevertheless, all topological results about sets and functions
definable in the language of rings augmented by the valuation map
remain true even if an angular component does not exist. Indeed,
since the $K$-topology induced by the valuation $v$ as well as
closure and interior operations are $\mathcal{L}$-definable, the
concept of continuity, Lipschitz continuity etc.\ are first order
properties. Therefore elementary extensions can be used in the
study of these properties. After replacing a given ground field
with an $\aleph_{1}$-saturated elementary extension, one will thus
have an angular component at hand.

\section{A version of the \L{}ojasiewicz inequality}

First we remind the reader that Henselian valued fields of
equicharacteristic zero admit quantifier elimination and, more
precisely, elimination of $K$-quantifiers in the language of
Denef--Pas (Pas~\cite{Pa1}. (In the case of non-algebraically
closed fields, passing to the three sorts with additional two
maps: the valuation $v$ and the residue map, is not sufficient.)
Next note that every archimedean ordered group $\Gamma$ (which of
course may be regarded as a subgroup of the additive group
$\mathbb{R}$ of real numbers) admits quantifier elimination in the
Presburger language $(<,+,-,0,1)$ with binary relation symbols
$\equiv_{n}$ for congruences modulo $n>1$, $n \in \mathbb{N}$,
where $1$ denotes the minimal positive element of $\Gamma$ if it
exists or $1=0$ otherwise. Thus we can apply quantifier
elimination in the $\Gamma$-sort whenever $K$ is a Henselian rank
valued field of equicharacteristic zero.

\vspace{1ex}

Here and in the next section, we shall still need the following
easy consequence of the closedness theorem.

\begin{proposition}\label{bound}
Let $f: A \to K$ be a continuous $\mathcal{L}$-definable function
on a closed bounded subset $A \subset K^{n}$. Then $f$ is a
bounded function, i.e.\ there is an $\alpha \in \Gamma$ such that
$v(f(x)) \geq \alpha$ for all $x \in A$. \hspace*{\fill}$\Box$
\end{proposition}

Now we can readily prove the following version of the
\L{}ojasiewicz inequality, which is a generalization of the one
from~\cite[Proposition~9.1]{Now}.

\begin{proposition}\label{Loj}
Let $f,g_{1},\ldots,g_{m}: A \to K$ be continuous
$\mathcal{L}$-definable functions on a closed (in the
$K$-topology) bounded subset $A$ of $R^{m}$. If
$$ \{ x \in A: g_{1}(x)= \ldots =g_{m}(x) =0 \} \subset \{ x \in A: f(x)=0 \}, $$
then there exist a positive integer $s$ and a constant $c \geq 0$
such that
$$ |f(x)|^{s} \leq c \cdot |(g_{1}(x), \ldots ,g_{m}(x))| $$
for all $x \in A$.
\end{proposition}

\begin{proof}
Put $g = (g_{1},\ldots,g_{m})$. It is easy to check that the set
$$ A_{\gamma} := \{ x \in A: \; v(f(x))=\gamma \} $$
is a closed $\mathcal{L}$-definable subset of $A$ for every
$\gamma \in \Gamma$. By the hypothesis and the closedness
theorem~\cite[Theorem~3.1]{Now}, the set $g(A_{\gamma})$ is a
closed $\mathcal{L}$-definable subset of $K^{m} \setminus \{ 0
\}$, $\gamma \in \Gamma$. The set $v(g(A_{\gamma}))$ is thus
bounded from above, i.e.\
$$ v(g(A_{\gamma})) \leq \alpha (\gamma) $$
for some $\alpha(\gamma) \in \Gamma$. By elimination of
$K$-quantifiers, the set
$$ \Lambda := \{ (v(f(x)),v(g(x))) \in \Gamma^{2}: \; x \in A \}
   \subset \{ (\gamma,\delta) \in \Gamma^{2}: \; \delta \leq \alpha(\gamma) \} $$
is a definable subset of $\Gamma^{2}$ in the Presburger language,
and thus it is described by a finite number of linear inequalities
and congruences. Hence
$$ \Lambda \, \cap  \{ (\gamma,\delta) \in \Gamma^{2}: \; \gamma > \gamma_{0} \}
   \subset \{ (\gamma,\delta) \in \Gamma^{2}: \; \delta \leq s \cdot \gamma \} $$
for a positive integer $s$ and some $\gamma_{0} \in \Gamma$. We
thus get
$$ v(g(x)) \leq s \cdot v(f(x)) \ \ \text{if} \ \ x \in A, \, v(f(x))> \gamma_{0}. $$
Again, by the hypothesis, we have
$$ g(\{ x \in A:  \ v(f(x)) \leq \gamma_{0} \}) \subset K^{m} \setminus \{ 0
   \}. $$
Therefore it follows from the closedness theorem that the set
$$ \{ v(g(x)) \in \Gamma: \ v(f(x)) \leq \gamma_{0} \} $$
is bounded from above, say, by a constant $\delta$. Hence and by
Proposition~\ref{bound}, we get
$$ s \cdot v(f(x)) - v(g(x)) \geq \beta := \min \, \{ 0, s \cdot \alpha - \delta \} $$
for all $x \in A$. This is the desired conclusion formulated in
terms of valuation with constant $c := \exp (- \beta)$.
\end{proof}

\section{Proofs of the main results}

\emph{Proof of Proposition~\ref{Hol}.} Apply Proposition~\ref{Loj}
to the functions
$$ f(x) - f(y) \ \ \text{and} \ \ g_{i}(x,y)=x_{i} - y_{i}, \
   i=1,\ldots,n. $$
\hspace*{\fill}$\Box$

\vspace{1ex}

\emph{Proof of Proposition~\ref{Lip}.} Let $\mathbb{P}^{1}(K)$
stand for the projective line over $K$. Define an
$\mathcal{L}$-definable subset $E$ of $A \times A \times
\mathbb{P}^{1}(K)$ by putting $(x,z,u) \in E$ iff
$$ u = \frac{f(x)-f(z)}{x_{i}-z_{i}} $$
for some $i=1,\ldots,n$ such that $x_{i} \neq z_{i}$ and,
moreover, the value $v(u)$ is the largest from among the values
$v(w_{j})$ of the fractions
$$ w_{j} = \frac{f(x)-f(z)}{x_{j}-z_{j}},  \ \  j=1,\ldots,n, $$
with $x_{j} \neq z_{j}$. Let $\overline{E}$ be the closure of $E$
in $A \times A \times \mathbb{P}^{1}(K)$. Denote by
$$ \phi: K^{n} \times K^{n} \times \mathbb{P}^{1}(K) \to \mathbb{P}^{1}(K) $$
the canonical projection. Again, by the closedness theorem, the
image $\phi (\overline{E})$ is a closed subset of
$\mathbb{P}^{1}(K)$. But by Proposition~\ref{bound}, the function
$f$ is bounded. Therefore, since $f$ is locally Lipschitz
continuous with a fixed constant $c$, it is not difficult to
deduce that $\overline{E}$ is actually a subset of $A \times A
\times K$. Thus the image $\phi (\overline{E})$ is a subset of $K$
and a closed subset of $\mathbb{P}^{1}(K)$. Hence $\phi
(\overline{E})$ is a bounded subset of $K$, i.e.\ $v(\phi
(\overline{E})) \geq \delta$ for some $\delta \in \Gamma$. Then $d
:= \exp (- \delta)
>0$ is the Lipschitz constant we are looking for.
\hspace*{\fill}$\Box$

\section{Some problems on continuous and Lipschitz continuous
functions definable over Henselian valued fields}

Fix a Henselian rank one valued field $K$ of equicharacteristic
zero. It is more delicate to examine Lipschitz continuous
definable functions whose domains are non-closed subsets of
$K^{n}$. Then the conclusion of global Lipschitz continuity must
be replaced by that of piecewise Lipschitz continuity. One of the
most fundamental questions is just the following problem
concerning the latter property.

\vspace{1ex}

{\bf Problem~1.}
\begin{em}
Consider a semi-algebraic (or $\mathcal{L}$-definable or, more
generally, definable in a suitable language) function $f: A \to K$
with $A \subset K^{n}$, and suppose that $f$ is locally Lipschitz
continuous with a Lipschitz constant $c$. Then is $f$ piecewise
Lipschitz continuous with possibly some larger constant $d$? In
other words, does there exist a finite semi-algebraic (or
$\mathcal{L}$-definable etc.) partition $\{ A_{1},\ldots, A_{s}
\}$ of $A$ such that the restriction of $f$ to each subset $A_{i}$
is Lipschitz continuous with some constant $d$? Further, what can
one say about a new constant $d$ in comparison to $d$? Is it
possible to take the same constant $d=c$?
\end{em}

\begin{remark}
Here a suitable language may be, in the first place, one linked
with so called \emph{topological systems} in the sense of van den
Dries~\cite{Dries}. That notion seems to be very useful in
relating definable topologies and model theoretical treatment. In
particular, the language of Denef--Pas can be translated into such
a language of topological system on a given Henselian valued field
$K$ with the topology induced by the valuation (by adding the
inverse images under the valuation and angular component map of
relations on the value group and residue field).
\end{remark}

Over the real number field $\mathbb{R}$, the affirmative answer to
Problem~1 for semi-algebraic and (globally) subanalytic functions
was given by Kurdyka~\cite{Kur} and
Paru\'si\'nski~\cite{Par1,Par2}; an o-minimal version over real
closed fields was presented by Paw\l{}ucki~\cite{Paw}. Moreover,
one can take $d = M c$ where the constant $M>1$ depends only on
the dimension $n$ of the ambient affine space. Finally, let us
mention that one of Kurdyka's results, namely
\cite[Corollary~C]{Kur}, was inspired by a question of Professor
\L{}ojasiewicz.

\vspace{1ex}

Over the $p$-adic number fields $\mathbb{Q}_{p}$ (or their finite
extensions), the affirmative answer for semi-algebraic and
subanalytic functions was given by
Cluckers--Comte--Loeser~\cite[Theorem~2.1]{CCL}. Their proof
relies on a certain compatible cell decomposition for the function
$f$. It takes into account some comparison of distances from the
centers of cells in the domain and image of $f$, which is made by
means of a Jacobian property for definable functions. This result
is crucial for the theory of local density and local metric
properties of $p$-adic definable sets, developed in the
paper~\cite{CCL-1}. But it would be rather difficult to directly
deduce from that paper itself how the constant $d$ depends on $c$.

\vspace{1ex}

In the $p$-adic case, however, it is possible to take the same
Lipschitz constant $d=c$, which is no longer true over the field
of real numbers. This was established by
Cluckers--Halupczok~\cite[Theorem~1]{C-Ha}. Their approach
combines the compatible cell decomposition mentioned above and a
kind of simultaneous piecewise approximation of $f$ and its
derivative by a "monomial with fractional exponent" and its
derivative.

\vspace{1ex}

Problem~1 seems to be open in the case where $K$ is an arbitrary
Henselian rank one valued field. Observe that in this case, too,
the sets definable in the language of Denef--Pas admit
decomposition into a finite number of cells "combed" by finitely
many congruences, as demonstrated in our
paper~\cite[Corollary~2.7]{Now} and recalled below. We begin with
the concept of a cell. Consider an $\mathcal{L}$-definable subset
$B$ of $K^{n} \times \Bbbk^{m}$, a positive integer $\nu$ and
three $\mathcal{L}$-definable functions
$$ a(x,\xi),b(x,\xi),c(x,\xi): B \to K. $$
For each $\xi \in \Bbbk^{m}$ set
$$ C(\xi) := \{ (x,y) \in K^{n}_{x} \times K_{y}: \ (x,\xi) \in B,
$$
$$ v(a(x,\xi)) \lhd_{1} v((y-c(x,\xi))^{\nu}) \lhd_{2} v(b(x,\xi)), \
   \overline{ac} (y-c(x,\xi)) = \xi_{1} \}, $$
where $\lhd_{1},\lhd_{2}$ stand for $<, \leq$ or no condition in
any occurrence. If the sets $C(\xi)$, $\xi \in \Bbbk^{m}$, are
pairwise disjoint, the union
$$ C := \bigcup_{\xi \in \Bbbk^{m}} C(\xi) $$
is called a cell in $K^{n} \times K$ with parameters $\xi$ and
center $c(x,\xi)$; $C(\xi)$ is called a fiber of the cell $C$.

\begin{proposition}\label{CD}
Every $\mathcal{L}$-definable subset $A$ of $K^{n} \times K$ is a
finite disjoint union of sets each of which is a subset
$$ F := \bigcup_{\xi \in \Bbbk^{m}} F(\xi) $$
of a cell $C$ with center $c(x,\xi)$:
$$ C := \bigcup_{\xi \in \Bbbk^{m}} C(\xi) $$
determined by finitely many congruences:
$$ F(\xi) = \left\{ (x,y) \in C(\xi): \, v\left( f_{i}(x,\xi) (y -
   c(x,\xi))^{k_{i}} \right) \equiv_{M} 0, \ i=1,\ldots,s \right\}, $$
where $f_{1},\ldots,f_{s}$ are $\mathcal{L}$-definable functions
and $k_{1},\ldots,k_{s},M \in \mathbb{N}$.
\end{proposition}

A subset $F \subset K^{n}$ of the form as above will be called a
\emph{combed cell} in $K^{n} \times K$. The conclusion of
Proposition~\ref{CD} may thus be rephrased as follows:

\vspace{1ex}

\begin{em}
Every $\mathcal{L}$-definable subset $A$ of $K^{n} \times K$ is a
finite disjoint union of combed cells. \hspace*{\fill}
\end{em}

\vspace{1ex}

We now turn to the problem of extending continuous rational
functions from an algebraic subvariety to a continuous rational
function on the ambient variety, which was solved in the
papers~\cite{K-N,Now}. The former paper deals with real and
$p$-adic varities and the latter with varieties over an arbitrary
Henselian rank one valued field of equicharacteristic zero.

\vspace{1ex}

Below we state the extension problem for Lipschitz continuous
rational functions, which is open as yet. It may be connected with
the open problem on extending rational functions of class
$\mathcal{C}^{p}$, $p \in \mathbb{N}$, posed at the end of
~\cite[Section~13]{Now}.

\vspace{1ex}

{\bf Problem~2.}
\begin{em}
Let $f: V \to K$ be a rational function on an algebraic subvariety
$V$ of $K^{n}$. Suppose that $f$ is Lipschitz continuous with a
constant $c$. Does $f$ extend to a rational function $F: K^{n} \to
K$ that is Lipschitz continuous with some constant $d$?  Further,
what can one say about a new constant $d$ in comparison to $d$? Is
it possible to take the same constant $d=c$?
\end{em}

\vspace{1ex}

Finally, we pose two problems concerning a Henselian analogue of
the Tietze--Urysohn extension theorem.

\vspace{1ex}

{\bf Problem~3.}
\begin{em}
Let $f: A \to K$ be an $\mathcal{L}$-definable function on a
closed subset $A$ of $K^{n}$ .

1) If $f$ is continuous, does there exist a continuous
$\mathcal{L}$-definable extension $F: K^{n} \to K$?

2) If $f$ is Lipschitz continuous with a constant $c$, does there
exist a Lipschitz continuous $\mathcal{L}$-definable extension $F:
K^{n} \to K$ with a constant $d$? Can one take the Lipschitz
constant $d = c$?
\end{em}

\vspace{1ex}

The affirmative answer to the first question is given in our
paper~\cite{Now1} being in preparation. We also expect an
affirmative answer to the second question, although we are
currently able to construct only a Lipschitz continuous extension
$F$ with the same constant $c$, but without ensuring its
definability. Note that in the realm of pure topology, a
non-archimedean analogue on extending continuous functions from an
ultraparacompact space into a complete metric space was
established by Ellis~\cite{El}.


\vspace{2ex}

\begin{small}
Institute of Mathematics

Faculty of Mathematics and Computer Science

Jagiellonian University


ul.~Profesora \L{}ojasiewicza 6

30-348 Krak\'{o}w, Poland

{\em E-mail address: nowak@im.uj.edu.pl}
\end{small}

\end{document}